\documentclass[12pt,reqno, oneside]{amsart}
\usepackage{pdfsync}  
\usepackage{geometry}
\usepackage{amssymb} % for \nmid
\usepackage{hyperref}

\newtheorem{thm}{Theorem}[section]

\newtheorem{lem}[thm]{Lemma}
\newtheorem{prop}[thm]{Proposition}
\newtheorem*{nlem}{Lemma}

\newcommand{\dpow}[2] {\frac{#1^#2}{#2!}}
\newcommand{\dpowp}[2] {\frac{#1^{#2}}{(#2)!}}
\newcommand{\h}{\mathbf{h}}
\renewcommand{\r}{\mathbf{r}}

\newcommand{\inv}{^{-1}}

\begin{document}
\title[Reciprocals and permutation enumeration]{Reciprocals of exponential polynomials and permutation enumeration}
\author{Ira M. Gessel$^*$}
\address{Department of Mathematics\\
   Brandeis University\\
   Waltham, MA 02453-2700}
\email{gessel@brandeis.edu}
\date{\today}
\thanks{$^*$Supported by a grant from the Simons Foundation (\#427060, Ira Gessel).}
\begin{abstract}
We show that the reciprocal of a partial sum of the alternating exponential series,
\begin{equation*}
\biggl(\sum_{n=0}^{2m-1} (-1)^n \dpow xn\biggr)\inv,
\end{equation*}
is the exponential generating function for permutations in which every increasing run has length congruent to 0 or 1 modulo $2m$. More generally we study polynomials whose reciprocals are exponential generating functions for permutations whose run lengths are restricted to certain congruence classes, and extend these results to noncommutative symmetric functions that count words with the same restrictions on run lengths.

\end{abstract}

\maketitle
\section{Introduction}
\thispagestyle{empty}
Since
\begin{equation}
\label{e-inv1}
\biggl(\sum_{n=0}^\infty (-1)^n \dpow xn\biggr)\inv
  =\sum_{n=0}^\infty \dpow xn,
\end{equation}
we might ask if anything interesting can be said about the coefficients of reciprocals of the partial sums of $\sum_n (-1)^n x^n/n!$. For partial sums with an odd number of terms there are negative coefficients, but for partial sums with an even number of terms  the reciprocals  have positive coefficients. We will give here a general result on permutation enumeration that as a special case
gives a combinatorial interpretation to the numbers $u_n$ defined by
\begin{equation*}
\biggl(\sum_{n=0}^{2m-1} (-1)^n \dpow xn\biggr)\inv = \sum_{n=0}^\infty u_n \dpow xn.
\end{equation*}
We will see that $u_n$ is the number of permutations of $[n]=\{1,2,\dots, n\}$ in which every increasing run has length congruent to 0 or 1 modulo $2m$. (The increasing runs of a permutation are its maximal increasing consecutive subsequences; thus, the increasing runs of the permutation $132576$ are $13$, $257$, and $6$, with lengths 2, 3 and~1.)

 For example, when $m=2$, we have
\[\sum_{n=0}^\infty u_n \dpow xn 
=\left( 1-x+\dpow x2  -\dpow x3\right)\inv,\]  and the first few values of $u_n$ are as follows:
\[\vbox{\halign{\ \hfil\strut$#$\hfil\ \vrule&&\hfil\ $#$\ \hfil\cr
n&0&1&2&3&4&5&6&7&8&9&10&11&12&13\cr
\noalign{\hrule}
u_n&1&1&1&1&2&10&50&210&840&4200&29400&231000&1755600&13213200\cr}}\]

More generally, we consider a class of polynomials whose reciprocals, as exponential generating functions, count permutations whose run lengths lie in certain congruence classes. Our exponential generating function results are derived from more general results on noncommutative symmetric functions that count words with the same restrictions on run lengths.

\section{Words and permutations with restricted run lengths}
\label{s-runs}
A simple class of polynomials with the property under consideration is given by the following result,  which we will prove later.
\begin{prop}
\label{p-egf}
Let $m$ and $r$ be  positive integers. Then
the coefficient of $x^n/n!$ in 
\begin{equation*}
\Biggl(\sum_{k=0}^{m-1} \left(\dpowp x{kr} - \dpowp x{kr+1}   \right)\Biggr)\inv
\end{equation*}
is the number of permutations of $[n]$ in which every increasing run has length congruent to $0, 1, \dots,$ or $r-1$ modulo $mr$.
\end{prop}

If we take the limit as $m\to\infty$ in Proposition \ref{p-egf}, we recover the known result \cite[pp.~156--157]{db} that 
the coefficient of $x^n/n!$ in 
\[\left(1-x+\dpow xr - \dpowp x{r+1} + \dpowp x{2r} -\dpowp x{2r+1}+\cdots\right)\inv\]
 is the number of permutations of $[n]$ in which every increasing run has length at most $r-1$. 
 
Proposition \ref{p-egf} is a consequence of a more refined result on noncommutative symmetric functions. 
We define the \emph{complete noncommutative symmetric function} $\h_n$ by 
\[
\mathbf{h}_{n}=\sum_{i_{1}\leq\cdots\leq i_{n}}X_{i_{1}}X_{i_{2}}\cdots X_{i_{n}},
\]
where the $X_i$ are noncommuting variables. (The algebra of noncommutative symmetric functions is generated by the $\h_n$. Noncommutative symmetric functions are studied extensively in \cite{gkllrt} and its successors, though with different notation and a different, but equivalent, definition.)

A \emph{composition} is a finite (possibly empty) sequence of positive integers. A \emph{composition of $n$} is a composition with sum $n$.
For a composition $L=(L_1,L_2,\dots, L_k)$, we define the \emph{ribbon noncommutative symmetric function} $\r_L$ by
\[
\mathbf{r}_{L}=\sum_{i_1,i_2,\dots, i_n}X_{i_{1}}X_{i_{2}}\cdots X_{i_{n}}
\]
where the sum is over all $i_{1},\dots,i_{n}$  satisfying
\begin{equation*}
\underset{L_{1}}{\underbrace{i_{1}\leq\cdots\leq i_{L_{1}}}}>\underset{L_{2}}{\underbrace{i_{L_{1}+1}\leq\cdots\leq i_{L_{1}+L_{2}}}}>\cdots>\underset{L_{k}}{\underbrace{i_{L_{1}+\cdots+L_{k-1}+1}\leq\cdots\leq i_{n}}}.
\end{equation*}
In other words, $\r_L$ is the sum of all words in the $X_i$ whose weakly increasing run lengths are $L_1,L_2,\dots, L_k$.

The connection between noncommutative symmetric functions and exponential generating functions is given by a well-known homomorphism $\Psi$ from noncommutative symmetric functions to power series in $x$: The homomorphism is determined  by $\Psi(\h_n)=x^n/n!$, but a more enlightening description of $\Psi$ is that if 
$f$ is a noncommutative symmetric function then the coefficient of $x^n/n!$ in $\Psi(f)$ is the coefficient of $x_1x_2 \cdots x_n$ in the result of replacing $X_1,X_2,\dots$ in $f$ with commuting variables $x_1, x_2, \dots$.
It follows that for a composition $L=(L_1,\dots, L_k)$ of $n$, we have
$\Psi(\r_L)=\beta(L)x^n/n!$, where $\beta(L)$ is the number of permutations of $[n]$ whose sequence of increasing run lengths is $L$. (See, e.g., \cite[Section 4.2]{gz}.)

We will state our main results for sums of  ribbon noncommutative symmetric functions and the corresponding exponential generating functions.

Our main tool is the ``run theorem" as stated by Zhuang \cite[Theorem 1]{zhuang}, which allows us to count words with restrictions on the lengths of their increasing runs.
The result was first proved in an equivalent form by Gessel \cite[Theorem 5.2]{thesis}; a closely related result was proved by Jackson and Aleliunas \cite[Theorem 4.1]{ja}.

\begin{nlem}[The run theorem]
\label{t-runs}
Suppose that sequences $a_n$ and $w_n$ are related by 
\begin{equation}
\biggl(\sum_{n=0}^{\infty}a_{n}x^{n}\biggr)^{-1}=\sum_{n=0}^{\infty}w_{n}x^{n},
\label{e-cx}
\end{equation}
where $a_0=w_{0}=1$. Then 
\[
\biggl(\sum_{n=0}^{\infty}a_{n}\mathbf{h}_{n}\biggr)^{-1}=\sum_{L}w_{L}\mathbf{r}_{L},
\]
where the sum on the right is over all compositions $L$, and $w_{L}=w_{L_{1}}\cdots w_{L_{k}}$ for $L=\left(L_{1},\dots,L_{k}\right)$.
\end{nlem}

We are interested in the case in which only finitely many of the $a_n$ are nonzero and each $w_n$ is either 0 or 1. The next proposition results from applying the run theorem to this case, and then applying the homomorphism $\Psi$.

\begin{prop}
\label{p-main}
Let $a(x)=a_0+a_1 x +\cdots+a_d x^d$, where $a_0=1$, and suppose that every power series coefficient of 
$1/a(x)$ is $0$ or $1$. Let $S$ be the set of positive integers $s$ such that the coefficient of $x^s$ in $1/a(x)$ is $1$. Then 
\begin{equation*}
\biggl(\sum_{n=0}^d a_n \h_n\biggr)\inv=\sum_{L}\r_L,
\end{equation*}
where the sum on the right is over all compositions $L$ with every part in $S$, and the coefficient of $x^n/n!$ in 
\begin{equation*}
\biggl(\sum_{n=0}^d a_n \dpow xn\biggr)\inv
\end{equation*}
is the number of permutations of $[n]$ in which every increasing run length is in $S$.
\end{prop}

A simple application of Proposition \ref{p-main} is the following.
\begin{prop}
\label{p-ncsf}
Let $m$ and $r$ be  positive integers. Then
\begin{equation*}
\biggl(\sum_{k=0}^{m-1} \left(\h_{kr} - \h_{kr+1}\right)\biggr)\inv
  = \sum_{L}\r_L,
\end{equation*}
where the sum on the right is over compositions $L$ in which every part is congruent to $0, 1, \dots,$ or $r-1$ modulo $mr$.
\end{prop}

\begin{proof}
In Proposition \ref{p-main} we take

\begin{equation*}
a(x)\inv= 
\frac{1 + x+x^2+\cdots +x^{r-1}}{1-x^{mr}}\\
  =\frac{1-x^r}{(1-x)(1-x^{mr})}
\end{equation*}
Then
\begin{align*}
a(x)&= (1-x)\frac{1-x^{mr}}{1-x^r}
    =(1-x)(1+x^r + x^{2r}+\cdots + x^{(m-1)r})\\
    &=\sum_{k=0}^{m-1} \left(x^{kr} - x^{kr+1}   \right).
    \qedhere
\end{align*}
\end{proof}

The limiting case $m\to \infty$ of Proposition \ref{p-ncsf} was given by Gessel \cite[Example 3, pp.~51--52]{thesis}. Results closely related  to this limiting case were 
proved by Jackson and Aleliunas \cite[Corollary 7.2]{ja} and by Yang and Zeilberger \cite{yz}.

Proposition \ref{p-egf} follows directly from Proposition \ref{p-ncsf} by applying the homomorphism $\Psi$.
By the same reasoning, we have a slight generalization. 

\begin{prop}
\label{p-gen}
Let $b$ and $r$ be  positive integers. Then 
\begin{equation*}
\biggl(\sum_{k=0}^{m-1} \left(\h_{krb} - \h_{(kr+1)b}\right)\biggr)\inv
  = \sum_{L}\r_L,
\end{equation*}
where the sum is over compositions $L$ in which every part is congruent to $0, b, \dots,$ or $(r-1)b$ modulo $mrb$.\qed
\end{prop}

\section{Necessary and sufficient conditions}

The results of Section \ref{s-runs} suggest the following problem: Find all \emph{finite} sequences $a_1, a_2,\dots, a_n$ such that in the ribbon expansion
\begin{equation}
\label{e-ribbon}
(1+a_1 \h_1 +\cdots +a_n\h_n)^{-1}=\sum_L w_L \r_L,
\end{equation}
each coefficient $w_L$ is 0 or 1. 
By the run theorem, 
this is equivalent to finding all   polynomials $a(x)$ (with constant term 1) such that every coefficient of $a(x)^{-1}$ is 0 or~1.

First we  give a necessary condition.
\begin{lem}
\label{l-bounded}
Let $R(x)$ be a rational function with bounded integer power series coefficients. 
Then $R(x) = N(x)/(1-x^m)$ for some polynomial $N(x)$ and some positive integer $m$.
\end{lem}
\begin{proof}
Let $R(x)=\sum_{n=0}^\infty r_n x^n$. Then the numbers $r_n$ satisfy a linear homogeneous recurrence with constant coefficients, so by the pigeonhole principle and the fact that the set $\{r_0, r_1, r_2,\dots\}$ is finite, the sequence $r_0, r_1,\dots$ 
must be eventually periodic, and the result follows.
\end{proof}

We can now give a necessary and sufficient condition. Let $\Phi_d(x)$ denote the $d$th cyclotomic polynomial. 
We will need three well-known properties of  cyclotomic polynomials:
\begin{enumerate}
\item[(i)] The polynomials $\Phi_d(x)$ are irreducible.
\item[(ii)] We have the factorization $x^m-1 = \prod_{d\mid m}\Phi_d(x).$
\item[(iii)] Every cyclotomic polynomial has leading coefficient 1.
\end{enumerate}

\begin{prop}
\label{p-necsuf}
Let $a(x)$ be a nonconstant polynomial. Then every power series coefficient of $a(x)^{-1}$ is 0 or 1 if and only  if for some integer $m$ we have
\begin{equation*}
a(x) = \frac{1-x^m}{N(x)}
\end{equation*}
where $N(x)$ is a product of distinct cyclotomic polynomials, $N(x)=\Phi_{d_1}(x)\cdots \Phi_{d_s}(x)$,  each $d_i$ divides $m$,  all coefficients of $N(x)$ are 0 or 1, and the degree of $N(x)$ is less than $m$.
\end{prop}

\begin{proof}
Suppose that $a(x)$ is a nonconstant polynomial such that every power series coefficient of $a(x)^{-1}$ is 0 or 1. Then by Lemma \ref{l-bounded}, $a(x)=(1-x^m)/N(x)$ for some polynomial $N(x)$. 

Since $a(x)$ is a polynomial, by (i) and (ii), $N(x)$ must be a product of cyclotomic polynomials $\pm \Phi_{d_1}(x)\cdots \Phi_{d_s}(x)$ for distinct divisors $d_1,\dots, d_s$ of $m$.

Since $a(x)$ is not a constant, $N(x)$ must have degree less than $m$.
Then
\begin{equation*}
a(x)^{-1}=\frac{N(x)}{1-x^m}=N(x) +x^m N(x) +x^{2m}N(x)+\cdots,
\end{equation*}
so since the degree of $N(x)$ is less than $m$, the coefficients of $N(x)$ are all 0 or 1. 
Finally,  since $N(x)$ has all nonnegative coefficients, by (iii) we must have the plus sign in $N(x)=\pm \Phi_{d_1}(x)\cdots \Phi_{d_s}(x)$.

The proof of the converse is easy, and is omitted.
\end{proof}

Proposition \ref{p-necsuf} is not completely satisfactory, since it does not give a simple criterion for determining when a product of cyclotomic polynomials has all coefficients 0 and 1, and there is probably  no simple criterion.

The simplest examples not included in Proposition \ref{p-gen} are 
\begin{equation*}
a(x) =1-x^2-x^3+x^4+x^5-x^7=(1-x)(1+x+x^2)(1-x^2+x^4),
\end{equation*}
for which 
\begin{equation*}
a(x)^{-1} = \frac{(1+x^2)(1+x^3)}{1-x^{12}}=\frac{1+x^2+x^3+x^5}{1-x^{12}}
\end{equation*}
and
\begin{equation*}
a(x) = 1-x^2-x^3+x^5+x^6-x^8=(1-x)(1+x)(1-x^3+x^6)
\end{equation*}
for which 
\begin{equation*}
a(x)^{-1}= \frac{(1+x^2+x^4)(1+x^3+x^6)}{1-x^{18}}
  =\frac{1+x^2+x^3+x^4+x^5+x^6+x^7+x^8+x^{10}}{1-x^{18}}.
\end{equation*}

Many examples of allowable $N(x)$, including all of our examples so far, can be obtained from products of the polynomials
\begin{equation*}
1+x^k + x^{2k}+\cdots +x^{(r-1)k}=\frac{1-x^{rk}}{1-x^k}
  =\prod_{\substack{d\mid rk\\d\nmid k}}\Phi_d(x).
\end{equation*}
Any such product will have positive coefficients, but there does not seem to be a simple criterion for determining when they are all 0 or 1.

Not all products of cyclotomic polynomials with all coefficients  0 or 1 are of this form. For example, we have
\begin{equation*}
\Phi_5(x)\Phi_6(x) = (1+x+x^2+x^3+x^4)(1-x+x^2) = 1+{x}^{2}+{x}^{3}+{x}^{4}+{x}^{6},\end{equation*}
which gives the polynomial
\begin{align*}
a(x) &= \frac{1-x^{30}}{1+{x}^{2}+{x}^{3}+{x}^{4}+{x}^{6}}\\
&= 1-{x}^{2}-{x}^{3}+2\,{x}^{5}+{x}^{6}-{x}^{7}-2\,{x}^{8}-{x}^{9}+2\,{x}
^{10}\\
&\qquad+2\,{x}^{11}-2\,{x}^{13}-2\,{x}^{14}+{x}^{15}+2\,{x}^{16}+{x}^{17
}-{x}^{18}-2\,{x}^{19}+{x}^{21}+{x}^{22}-{x}^{24}
\end{align*}
and
\begin{equation*}
a(x)\inv = \frac{1+{x}^{2}+{x}^{3}+{x}^{4}+{x}^{6}}{1-x^{30}}.
\end{equation*}

Combining Proposition  \ref{p-main} with Proposition \ref{p-necsuf} gives a more explicit version of Proposition  \ref{p-main}.

\begin{thm}
\label{t-main2}
Let $a(x)=a_0+a_1 x +\cdots+a_d x^d$, where $a_0=1$, and suppose that every power series coefficient of 
$1/a(x)$ is $0$ or $1$. Then there exists a positive integer $m$ and 
a finite set $T\subseteq \{0,1,\dots, m-1\}$, with $0\in T$,  such that
\begin{equation*}
a(x)^{-1}=\sum_{t\in T^*} x^t,
\end{equation*}
where $T^*$ is the set of nonnegative integers congruent modulo $m$ to an element of $T$.

Moreover
\begin{equation*}
\biggl(\sum_{n=0}^d a_n \h_n\biggr)\inv=\sum_{L}\r_L,
\end{equation*}
where the sum on the right is over all compositions $L$ with every part in $T^*$, and the coefficient of $x^n/n!$ in 
\begin{equation*}
\biggl(\sum_{n=0}^d a_n \dpow xn\biggr)\inv
\end{equation*}
is the number of permutations of $[n]$ in which every increasing run length is in $T^*$.
\end{thm}

\bigskip
\noindent\textbf{Acknowledgment.} I would like to thank Brendon Rhoades for discussions that ultimately led to the results of this paper, and Yan Zhuang and two anonymous referees for helpful comments on an earlier version.


\begin{thebibliography}{9}
\bibitem{db}
F. N. David and D. E. Barton, \emph{Combinatorial Chance.} Hafner Publishing Co., New York, 1962.

\bibitem{gkllrt}
Israel M. Gelfand, Daniel Krob, Alain Lascoux, Bernard Leclerc, Vladimir S. Retakh, and Jean-Yves Thibon, \emph{Noncommutative symmetric functions}, Adv. Math. 112 (1995),  218--348.

\bibitem{thesis}
Ira Martin Gessel, \emph{Generating Functions and Enumeration of Sequences}, Ph.D. thesis, Massachusetts Institute of Technology, 1977.

\bibitem{gz}
Ira M. Gessel and Yan Zhuang, 
\emph{Counting permutations by alternating descents} 
Electron. J. Combin. 21 (2014), no. 4, Paper 4.23, 21 pp. 

\bibitem{ja}
D. M. Jackson and R. Aleliunas,
\emph{Decomposition based generating functions for sequences},
Canad. J. Math. 29 (1977),  971--1009. 
 
\bibitem{yz}
Mingjia Yang and Doron Zeilberger, Increasing consecutive patterns in words,
\href{https://arxiv.org/abs/1805.06077}{\texttt{arXiv:1805.06077 [math.CO]}}, 2018.

\bibitem{zhuang}
Yan Zhuang,
\emph{Counting permutations by runs},
J. Combin. Theory Ser. A 142 (2016), 147--176. 
\end{thebibliography}
\end{document}